\documentclass[12pt, reqno]{amsart}

\usepackage{amssymb,latexsym,amsmath,amsfonts}
\usepackage{mathrsfs}
\usepackage{graphicx}
\usepackage[usenames]{color}
\usepackage{hyperref}
\usepackage{comment}
\usepackage{enumitem}

\definecolor{DPurple}{rgb}{0.46,0.2,0.69}

\numberwithin{equation}{section}

\allowdisplaybreaks

\theoremstyle{definition}
\newtheorem{definition}{Definition}[section]

\theoremstyle{remark}
\newtheorem{remark}[definition]{Remark}

 \theoremstyle{plain}
\newtheorem{theorem}[definition]{Theorem}
\newtheorem{result}[definition]{Result}

\newtheorem{lemma}[definition]{Lemma}



\begin{document}

\title[Sharing of Moving Hyperplanes Wandering on $\mathbb{P}^n$]{Normal Families of Holomorphic Curves and Sharing of Moving 
Hyperplanes Wandering on $\mathbb{P}^n$}

\author{Gopal Datt}
\address{Department of Mathematics, Babasaheb Bhimrao Ambedkar Univeristy, Lucknow, India}
\email{ggopal.datt@gmail.com, gopal.du@bbau.ac.in}

\author{Naveen Gupta}
\address{Department of Mathematics, Indian Institute of Science, Bengaluru, India}
\email{naveengupta@iisc.ac.in}

\author{Nikhil Khanna}
\address{Department of Mathematics, Sultan Qaboos University, Muscat, Oman}
\email{n.khanna@squ.edu.om}

\author{Ritesh Pal}
\address{Department of Mathematics, Babasaheb Bhimrao Ambedkar Univeristy, Lucknow, India}
\email{rriteshpal@gmail.com}



\begin{abstract}
In this paper, we extend a result of Schwick concerning normality and sharing values in one complex variable for families of holomorphic curves
taking values in $\mathbb{P}^n$.  We consider wandering moving hyperplanes (i.e., depending on the respective holomorphic curve in the family
under consideration), and establish a sufficient condition of normality concerning shared hyperplanes. 
\end{abstract}
\maketitle

\vspace{-0.6cm}
\section{Introduction and main results}\label{S:intro}

The main result of this paper is influenced by two results on normality in one complex variable 
 and their extensions in several complex variables. The first is the famous, Montel's theorem which is  also known as the fundamental normality test, and the
second is a result of Schwick, which is perhaps the first result that draws connection between normality and sharing values. Let us begin with the first 
influence: The classical Montel's theorem asserts that a family of  meromorphic mappings on a 
planar domain is normal if this family
omits three fixed, distinct points in the extended complex plane $\mathbb{C}\cup\{\infty\}$. A natural approach to 
generalizing the Montel's theorem is to take into account  exceptional (or omitted) values depending on the respective function in the family
under consideration, as opposed to keeping fixed exceptional values. This approach was studied by Carath\'eodory, and he established the following result
which is known as the Montel\,--\,Carath\'eodory theorem.

\begin{result}\label{R: MC Thm}\cite[Page no.\,202]{Caratheodory 60}
Let $\mathcal{F}$ be a family of functions meromorphic on a planar domain $D\subset\mathbb{C}$. Suppose that for  each 
$f\in\mathcal{F}$ there exists three distinct values $a_f, b_f, c_f$ in $\mathbb{C}\cup\{\infty\}$ such that $f$ omits values $a_f, b_f, c_f$.
Then the family $\mathcal{F}$ is normal provided the product of chordal distances of these points is greater than  a fixed positive 
number $\varepsilon>0$, for all $f\in \mathcal{F}$, i.e., 
$$\chi(a_f, b_f)\cdot \chi(b_f, c_f)\cdot \chi(a_f, c_f)> \varepsilon, \forall f\in\mathcal{F},$$ 
where $\chi$ denotes the chordal metric. 
\end{result}
 
Before stating the second result that influences this work, let us recall the meaning of sharing of values. Let $f$ and $g$ be two non-constant meromorphic 
functions and $\alpha\in\mathbb{C}\cup\{\infty\}$. We say $f$ and $g$ share the value $\alpha$ (ignoring multiplicities) if 
$ f^{-1}(\{\alpha\})=g^{-1}(\{\alpha\})$
as sets. (Note that $f^{-1}(\{\infty\})$ denotes the set of poles of $f$.) Schwick established the following result wherein he connects normality with
the notion of shared values. 

\begin{result}\label{R: Schwick}\cite{Schwick 92}
Let $\mathcal{ F}$ be a family of functions meromorphic on a planar domain $D\subset\mathbb{C}$,  and 
$a, b, c$ be three distinct fixed complex numbers. If for each $f\in\mathcal{F}$,  $f$ and 
its first derivative $f'$ share the values $a, b, c$, then $\mathcal{F}$ is normal on  $D$.
\end{result}

The Montel\,--\,Carath\'eodory theorem  and Result~\ref{R: Schwick} have been generalized for families of holomorphic
 mappings taking values in the 
$n$-dimensional complex projective 
space $\mathbb{P}^n$ (see \cite{Dufresnoy, YPY 15, Datt 24}). To understand these extensions and our main theorem, we need to recall some basic definitions and facts. 
We, first, recall a quantity\,---\,for a given  collection of $q\geq n+1$ hyperplanes in general position in $\mathbb{P}^n$\,---\,which 
quantifies in a natural way
to what extent this collection is in general position. 
To this end, we 
fix a system of homogeneous coordinates $w=[w_0\,:\,\dots\,:\,w_n]$ in $\mathbb{P}^n$, then 
any hyperplane $H$ in $\mathbb{P}^n$ is given by
\begin{equation}\label{eq: Hyperplane}
H:=\{[w_0\;:\,w_1\,:\;\dots\,:\,w_n]\in\mathbb{P}^n\,\big|\, a_0w_0+a_1w_1+\dots+a_nw_n=0\},
\end{equation}
where $(a_0,\dots,a_n)=\alpha\in\mathbb{C}^{n+1}$  is a non-zero vector. Let $H_1,\dots, H_{n+1}$ be $n+1$ hyperplanes in 
$\mathbb{P}^n$, and let $\alpha_j=(a_0^{(j)},\dots,a_n^{(j)})$ be non-zero vectors in $\mathbb{C}^{n+1}$ such that for each $j\in\{1, \dots, n+~1~\}$,
$H_j$ is given by the homogenous polynomial $a_0^{(j)}w_0+\dots+a_n^{(j)}w_n$ as in \eqref{eq: Hyperplane}. 
We define
\begin{equation*}
  {D}(H_1,\dots, H_{n+1}) :=
  \left|\det\big[a^{(j)}_l\big]_{1\leq j\leq n+1, \atop 0\leq l\leq n}\right|.
\end{equation*}
Note that the quantity $\mathcal{D}(H_1,\dots, H_{n+1})$ depends only on 
the hyperplanes $\{H_j\,:\,1\leq j\leq n+1\}$ and is independent of the choice of $\alpha_j$. Now, let $H_1,\dots, H_q$ be $q\geq n+1$ hyperplanes 
in $\mathbb{P}^n$. Set
\begin{equation*}
 \mathcal{D}(H_1,\dots, H_q) := 
  \prod_{1\leq j_1<\dots<j_{n+1}\leq q}{D}(H_{j_1},\dots, H_{j_{n+1}}).
\end{equation*}
We say that hyperplanes $H_1, \dots, H_q$ are in general position if $D(H_1,\dots, H_q)>0.$
\smallskip

We can now state the analogue of  Montel's Theorem in higher dimensional settings:
{\it A family of $\mathbb{P}^n$-valued  holomorphic  mappings on a domain $D\subseteq\mathbb{C}^m$
is normal if this family omits $2n+1$ hyperplanes in general position in $\mathbb{P}^n$}. Note that in this extension of the Montel's
 Theorem,
distinct values in $\mathbb{C}\cup\{\infty\}$ are replaced by the hyperplanes in general position in $\mathbb{P}^n$. To get 
an extension of Result~\ref{R: Schwick}, we need to extend the notion of derivative. This was done by Ye, Pang and Yang in \cite{YPY 15} for 
$\mathbb{P}^n$-valued holomorphic curves. Let us first see a few things about  a  
$\mathbb{P}^n$-valued holomorphic curves.
\smallskip

 In what follows in this section we fix $D\subset \mathbb{C}$ as a domain.
 Let $f:D\to \mathbb{P}^n$ be a holomorphic mapping. Fixing a system of homogeneous coordinates on
$\mathbb{P}^n$, for each $a\in D$, we have a holomorphic mapping $\widehat{f}(z):=(f_0(z), \dots, f_n(z))$ on some neighbourhood $U$ of $a$ such that
$\{z\in U\,\big|\, f_0(z)=f_1(z)=\dots=f_n(z)=0\}=\emptyset$ and $f(z)=[f_0(z)\,:\,f_1(z)\,:\,\dots\,:\,f_n(z)]$ for each $z\in U$. Any such holomorphic 
mapping $\widehat{f}:U\to \mathbb{C}^{n+1}$ is called a {\it reduced representation} of $f$ on $U$. A holomorphic mapping $f:D\to \mathbb{P}^n$ 
is called a {\bf holomorphic curve}. It is good to mention here that if for two open neighbourhoods, say $U_1$ and $U_2$, $\widehat{f}_1$ and $\widehat{f}_2$
are two reduced representations of $f$, and $U_1\cap U_2\neq \emptyset$ then there is a non-vanishing holomorphic function $h:U_1\cap U_2\to\mathbb{C}$
such that $\widehat{f}_2=h\widehat{f}_1$ on $U_1\cap U_2$.
\smallskip

Given a hyperplane $H$ as in \eqref{eq: Hyperplane}, and a reduced representation $\widehat f$ of a 
 holomorphic curve $f$, define a new holomorphic function 
\begin{equation}\label{eq: f h function}
\langle \widehat{f}, H\rangle:=\sum_{l=0}^na_i f_i(z),
\end{equation}
and put $$\|\widehat{f}(z)\|:=\max_{0\leq l\leq n}|f_i(z)|, \quad \text{ and } \quad  \|H\|:=\max_{0\leq l\leq n}|a_l|.$$
We shall use $f$ instead of $\widehat{f}$ in \eqref{eq: f h function}, when there is a possibility that
expressions are independent of the choice of reduced representations.
\smallskip

We  now prepare the recipe for extending the notion of derivative. To this end, take a holomorphic curve $f=[f_0\,:\,\dots\,:\,f_n]$ 
defined in 
$D\subset \mathbb{C}$ with $f_0\not\equiv 0$, and $d(z)$ be a function holomorphic in $D$ such that 
$$f_0^2/d \text{ and }W(f_0, f_l)/d \quad (l=1,\dots, n)$$
are holomorphic functions without common zeros. Here $W(f_0, f_l)$ denotes  the Wronskian of $f_0$ and $f_l$, i.e., $W(f_0, f_l)=f_0f_l'-f_0'f_l$. 
\begin{definition}\cite{YPY 15}
A holomorphic map induced by the map 
$$D\ni z\mapsto (f_0^2, W(f_0, f_1),\dots,W(f_0, f_n))\in \mathbb{C}^{n+1}$$
is called a {\bf derived holomorphic  map} of $f$ and is written as 
$$\nabla f=[f_0^2/d\,:\,W(f_0, f_1)/d\,:\,\dots\,:\,W(f_0, f_n)/d].$$
\end{definition}

We remark here that for $n=1$, a holomorphic curve $f=[f_0\,:\,f_1]$ is nothing but 
the meromorphic function $f_1/f_0,$ and the derived holomorphic map
 $\nabla f$ corresponds exactly to the derivative of the meromorphic function $f_1/f_0$.
\smallskip

Ye et.\,al., in \cite{YPY 15}, generalized Result~\ref{R: Schwick} for a family of holomorphic curves in the following manner. 
\begin{result}\cite{YPY 15}\label{R: YPY}
Let $\mathcal{F}$ be a family of holomorphic curves defined on a planar domain $D$ into $\mathbb{P}^n$, $H_1,\dots, H_{2n+1}$
 be hyperplanes in general 
position in $\mathbb{P}^n$, and $\delta$ be a real number in the open interval $(0,1)$. Suppose that for each $f\in\mathcal{F}$,
\begin{enumerate}
\item[{\em(1)}] the maps $f$ and $\nabla f$ share $H_j$ on $D$ for $j=1, \dots, 2n+1$.
\item[{\em(2)}] if $f(z)\in \displaystyle\cup_{j=1}^{2n+1} H_j,$ then $\displaystyle\frac{|\langle f(z), H_0\rangle|}{\|f(z)\|\cdot\|H_0\|}\geq \delta$, where
$H_0=\{x_0=0\}$.
\end{enumerate}  
Then the family $\mathcal{F}$ is normal on $D$.
\end{result} 

Now there arise two natural directions to generalize Result~\ref{R: YPY}:
\begin{enumerate}
\item[(a)]  Instead of fixed shared hyperplanes, one may consider shared hyperplanes depending on the respective holomorphic curve in the family
under consideration, kind of wandering hyperplanes. 
In this scenario, of course, one can hope for normality under stronger assumption on the {\bf general positions} of the hyperplanes. It was seen in the statement of the Montel\,--\,Carath\'eodory theorem, wherein the exceptional values $a_f, b_f, c_f$ are not merely distinct but are
 {\bf uniformly distinct} in some sense. 
\item[(b)] Instead of  taking merely hyperplanes one may take moving hyperplanes with some extra assumptions. 
\end{enumerate}

In this paper we unify these two directions by considering moving hyperplanes which depend on the respective holomorphic curve in the given family, that is,
kind of {\bf wandering moving hyperplanes}. We defer the definition of moving hyperplane and some other related notions to the Section~\ref{def}. 
We are now in a position to state our main theorem, which is as follows.
\begin{theorem}\label{T: MainTheorem}
Let $\mathcal{F}$ be a family of functions holomoprhic in a domain $D\subset \mathbb{C}$ into $\mathbb{P}^n$. 
Let $\varepsilon$ be a positive real number in the open interval $(0, 1)$. Assume that 
for each $f\in \mathcal{F}$ there exists $2n+1$ moving hyperplanes $H_{1,f}, \dots, H_{(2n+1),f}$ such that $\{\widetilde{H_{j,f}}\,:\, f\in \mathcal{F}\}$
$(j=1, 2,\dots 2n+1)$ are normal families and there exists a positive constant $\delta>0$ satisfying
\begin{equation*}
\mathcal{D}(H_{1,f}, \dots, H_{(2n+1),f})(z)>\delta, \qquad \text{ for all } \ z\in D, f\in \mathcal{F}.
\end{equation*}
Suppose that for each $f\in \mathcal{F}$ 
\begin{enumerate}
\item[\em(1)] $\bigtriangledown f(z)\in H_{j,f}$ if and only if $f(z)\in H_{j,f}$ for all $j=1, \dots, 2n+1$.
\item[\em(2)] If $f(z)\in \bigcup_{j=1}^{2n+1} H_{j,f}$, then $|f_0(z)|\geq \varepsilon \|f(z)\|$.
\end{enumerate}
Then $\mathcal{F}$ is normal in $D$.
\end{theorem}

The object $\widetilde{H_{j,f}}$ refers to certain holomorphic function related to the moving hyperplane 
$H_{j,f}$, whose precise definition is given in Section~\ref{def}.

\section{Basic notions}\label{def}

In this section, we shall explain the concepts and terms that made an appearance
in Section~\ref{S:intro}, and to introducing certain basic notions needed in our proofs. In this section,
$D\subset \mathbb{C}$ will always denote a planar domain.
\smallskip

Let $z\in D, \ j\in\{1, \dots, 2n+1\}$, define $$L_j(z)(w):=a_{0}^{(j)}(z)w_0+a_{1}^{(j)}(z)w_1+\dots+a_{n}^{(j)}(z)w_n,$$ 
where $w=(w_0,w_1,\dots, w_n)
\in\mathbb{C}^{n+1}$, and $a_{l}^{(j)}(z)$ $(0\leq l\leq n)$ are holomorphic functions on $D$ without common zeros. For any fixed 
$z\in D$, let $\ker\,(L_j(z))$ denote the kernel, in $\mathbb{C}^{n+1}$,
 of the surjective linear functional $L_j(z)\,:\,\mathbb{C}^{n+1}\to \mathbb{C}$. If $\rho:\mathbb{C}^{n+1}\setminus\{0\}\to \mathbb{P}^n$
is the standard projective mapping,  then 
$$H_j(z):=\rho(\ker\,(L_j(z))\setminus\{0\}) (\subset \mathbb{P}^n)$$
is the moving hyperplane corresponding to the linear form $L_j(z)$. If there is no confusion, we write $H_j$ instead of $H_j(z)$. Let $H_j\, 
(j=1\dots, {2n+1})$ be moving hyperplanes corresponding to the linear forms $L_j(z):=a_{0}^{(j)}(z)w_0+a_{1}^{(j)}(z)w_1+\dots+a_{n}^{(j)}(z)w_n$,
$j=1, \dots, 2n+1$. Define 
$$\mathcal{D}(H_1,\dots,H_{2n+1}):=\prod_{A\subset\{1,\dots, 2n+1\},|A|=n+1}\left(\sum_{\{j_0,\dots,j_n\}\subset A}\Big|
\det\big(a^{(j_l)}_{i}\big)_{0\leq l,i\leq n}\Big|\right).$$
The moving hyperplanes $H_1,\dots, H_{2n+1}$ are said to be in {\bf pointwise general position} if for every $z_0\in D$, the fixed 
hyperplanes $H_1(z_0),\dots, H_{2n+1}(z_0)$ are in general position. It is straightforward to see that 
the moving hyperplanes $H_1, \dots, H_{2n+1}$ are in pointwise general position in $\mathbb{P}^n$ if and only if 
$\mathcal{D}(H_1,\dots,H_{2n+1})(z)>0$ for all $z\in D$.
\smallskip

We now define the meaning of sharing hyperplane: Given two holomorphic curves  $f, g$ defined on $D$  into $\mathbb{P}^n$
and given a hyperplane $H$ in $\mathbb{P}^n$, we say the curves $f$ and $g$ share  the 
hyperplane $H$ if $f^{-1}(H)=g^{-1}(H)$ as sets and 
$f=g$ on $f^{-1}(H)$.

\smallskip
Let us now give the precise definition of the object $\widetilde{H_j}$ alluded to at the end of Section~\ref{S:intro}. Given a 
moving hyperplane $H$, in $\mathbb{P}^n$,  corresponding to the linear forms 
$L(z):=a_{0}(z)w_0+a_{1}(z)w_1+\dots+a_{n}(z)w_n$, we define a holomorphic mapping 
$\widetilde{H}$ on $D$ into $\mathbb{P}^n$ with the
reduced representation $(a_0(z),\dots, a_n(z))$, i.e., $\widetilde{H}(z)=[a_0(z)\,:\,\dots\,:\,a_n(z)], z\in D$.
\smallskip

Finally we end this section by giving the definition that is central to the discussion in Section~\ref{S:intro}. 
 The space $\operatorname{Hol}\,(D,\mathbb{P}^n)$ of holomorphic mappings
from $D$ into $\mathbb{P}^n$ is endowed with the compact-open topology.
\begin{definition}
A family $\mathcal{F}\subset \operatorname{Hol}\,(D,\mathbb{P}^n)$ is said to be normal if $\mathcal{F}$ is relatively compact in
$ \operatorname{Hol}\,(D,\mathbb{P}^n)$.
\end{definition}
It is worth noting that a sequence $\{f_\nu\}_{\nu=1}^\infty\subset \operatorname{Hol}\,(D,\mathbb{P}^n)$ converges uniformly to a holomorphic mapping 
$f$ on  compact subsets of $D$ if and only if  for any $a\in D$ , each $f_\nu$ has a reduced representation 
$\widehat{f}_\nu=(f_{\nu_0}, f_{\nu_1},\dots, f_{\nu_n})$ on some fixed open set $U\ni a$ in $D$ such that each sequence 
$\{f_{\nu_l}\}_{\nu=1}^\infty, (l=0, 1, \dots, n)$, converges to a holomorphic function $f_l$, $l=0,1,\dots, n$, on $U$ with the property that
$\widehat{f}=(f_0, f_1, \dots, f_n)$ is a reduced representation of $f$ on $U$.

\section{Essential Lemma}\label{S: EL}
In this section, we state two results that will be used to prove our main theorem. 
\smallskip

Zalcman's rescaling lemma is one of the well-known tools in the theory of normality. Loosely speaking, it asserts that in the absence of normality
a certain kind of infinitesimal convergence takes place. Zalcman's result has been studied widely and extended for various kind of families of functions 
(see \cite{AladroKrantz,Datt 19,Thai Trang Huong 03}). 
 While the maps of interest are univariate in this paper, they take values in the complex 
projective space $\mathbb{P}^n$ for $n\geq 2$, we need a higher-dimensional analogue of rescaling lemma of Zalcman, which is due to
Aladro and Krantz, and is stated as follows:

\begin{lemma}[{a special case of \cite[Theorem 3.1]{AladroKrantz}}]\label{L:Zalcman}
 Let $\mathcal{F}$ a family of holomorphic mappings 
 of a domain $D\subset \mathbb{C}^m$ into $\mathbb{P}^n$. The family $\mathcal{F}$ is not normal 
  if and only if there exist
  \begin{enumerate}
    \item[$(a)$] a point $z_0\in D$ and a sequence $\{z_\nu\}\subset D$ such that  $z_\nu\rightarrow z_0$;
    
    \item[$(b)$] a sequence $\{f_\nu\}\subset \mathcal{F}$;
    
    \item[$(c)$] a sequence $\{r_\nu\}\subset \mathbb{R}$ with $r_\nu>0$ and $r_\nu\rightarrow 0$; and
    
  \end{enumerate} 
  such that $ g_\nu(\zeta):= f_j(z_{\nu}+r_{\nu}\zeta)$, where $\zeta\in\mathbb{C}$ satisfies
  $z_{\nu}+r_{\nu}\zeta\in D$, converges uniformly on compact subsets of $\mathbb{C}$ to a 
  non-constant holomorphic mapping $h: \mathbb{C}\rightarrow \mathbb{P}^n.$
\end{lemma}

\begin{remark}
  We should mention here that the result of Aladro\,--\,Krantz in \cite{AladroKrantz} is more general 
  than Lemma~\ref{L:Zalcman}. There is a case missing,  wherein $M$ is {\it non-compact}, from
  their analysis. The correct  arguments were provided by
  \cite[Theorem 2.5]{Thai Trang Huong 03}. At any rate, Lemma~\ref{L:Zalcman} is the version of the above-mentioned result of 
Aladro and Krnatz that we need.
\end{remark} 

The following result, due to Green, is about the generalization of Picard's little theorem to several variables. The classical  Picard's theorem
asserts that a meromorphic function in $\mathbb{C}$ is constant if it omits three distinct values in $\mathbb{C}\cup\{\infty\}$. 

\begin{lemma}[{\cite[Corollary 2]{Green}}]\label{L:Green}
A holomorphic map $f:\mathbb{C}^m\to \mathbb{P}^n$ which omits $2n+1$ hyperplanes in general position in $\mathbb{P}^n$ is constant. 
\end{lemma}

\section{Proof of Main Theorem}
\begin{proof}[Proof of Theorem~\ref{T: MainTheorem}]
Let $\{f_\nu\}_{\nu=1}^\infty \subset \mathcal{F}$ be an arbitrarily chosen sequence. We aim to show that $\{f_\nu\}$ has a subsequence that
converges uniformly on compact subsets of $D$.  By the hypothesis of theorem, the families $\{\widetilde{H_{j,f}}\,:\, f\in \mathcal{F}\}$,
$j=1, 2,\dots, 2n+1$, that are induced from the collection of hyperplanes are normal. We assume, without loss of generality, that for each 
$j\in\{1,\dots, 2n+1\}$, the sequence $\{\widetilde{H_{j,f_\nu}}\}$ converges uniformly on compact subsets of $D$ to a holomorphic mapping, say $h_j$.
Let $\widehat{h_j}(z)=(a_{j_0}(z),\dots,a_{j_n}(z))$ be a reduced representation of  $h_j$, $j=1, \dots, 2n+1$. 
Let $H_j$, $j=1, \dots, 2n+1,$ be moving hyperplanes given by the 
homogeneous polynomials $a_{j_0}(z)w_0+\dots+a_{j_n}(z)w_n$, $j=1,\dots, 2n+1$. By the hypothesis we have
\begin{equation*}
\mathcal{D}(H_{1,f_v}, \dots, H_{(2n+1),f_\nu})(z)>\delta, \qquad \text{ for all } \ z\in D, \nu\geq 1,
\end{equation*}
therefore, we can claim that the limiting hyperplanes $H_1,\dots, H_{2n+1}$ are in general position in $\mathbb{P}^n$.
\smallskip

We now assume on the contrary that the sequence $\{f_\nu\}$, chosen in the previous paragraph, is not normal in $D$. Therefore, by 
Lemma~\ref{L:Zalcman} there exist:
\begin{itemize}
\item a subsequence of $\{f_\nu\}$, which we again denote by $\{f_\nu\}$ after renumbering,
\item a sequence of points $\{z_\nu\}\subset D$ such that $z_\nu\to z_0$, 
\item a sequence of positive real numbers $\{\rho_\nu\}\subset (0,+\infty)$ with
$\rho_\nu\to 0$,
\end{itemize}
 such that\,---\,defining the maps $g_\nu\,:\, \zeta\mapsto f_\nu(z_\nu+\rho_\nu\zeta)$ on suitable neighbourhoods of 
$0\in\mathbb{C}$\,---\,$\{g_\nu\}$  converges uniformly on compact subsets of $\mathbb{C}$ to a {\it non-constant} holomorphic map
$g\,:\,\mathbb{C}\to \mathbb{P}^n$. Then there exist reduced representations $\widehat{g}=(g_0,\dots, g_n)$ of $g$, and
 $\widehat{g_\nu}=(g_{\nu_0},\dots, g_{\nu_n})$, $\nu\geq 1$ of $g_\nu$, $\nu\geq 1$ such that $\{g_{\nu_l}\}$ converges uniformly
on compact subsets of $\mathbb{C}$ to $g_l$, for $l=0, \dots, n$. Which further implies\,---\, in view of 
the normality of $\{\widetilde{H_{j,f}}\,:\, f\in \mathcal{F}\}$\,---\,that 
$\langle \widehat{g_\nu}(\zeta), {H_{j,f\nu}}(z_\nu+\rho_\nu\zeta) \rangle$ converges uniformly on compact subsets of $\mathbb{C}$ to
$\langle \widehat{g}(\zeta), {H_{j}}(z_0) \rangle$. Since $H_j$, $j=1, 2,\dots, 2n+1$, are moving hyperplanes in general position in $\mathbb{P}^n$,
we infer that $H_j(z_0)$ $j=1, 2,\dots, 2n+1,$ are fixed hyperplanes in general position in $\mathbb{P}^n$.
\smallskip

Since $g$ is non-constant and $H_1(z_0), \dots, H_{2n+1}(z_0)$ are in general position,
 by Lemma~\ref{L:Green}, there exist at least one hyperplane $H_j(z_0)$, say $H_1(z_0)$, such that 
$\langle g(\zeta), H_1(z_0)\rangle\not \equiv 0$, and has some zeros in $\mathbb{C}$. Recall that 
$$\langle g(\zeta), H_1(z_0)\rangle=a_{j_0}(z_0)g_0(\zeta)+\dots+a_{j_n}(z_0)g_n(\zeta), \zeta \in\mathbb{C}.$$
Let $\zeta_0$ be a zero of $\langle g(\zeta), H_1(z_0)\rangle$, that is,
\begin{equation}\label{eq: g zero}
\langle g(\zeta_0), H_1(z_0)\rangle=0.
\end{equation}
 Since zeros of holomorphic 
functions are isolated, there exists a neighbourhood $N$ of $\zeta_0$ such that $\langle g(\zeta), H_1(z_0)\rangle\neq 0$ for all
$\zeta\in N\setminus\{\zeta_0\}$. We may, if necessary, 
shrink the neighbourhood $N$ so that  the reduced representation $\widehat g_\nu$ of $g_\nu$ is given by
$$\widehat{g_\nu}(\zeta)=(g_{\nu_0}(\zeta),\dots, g_{\nu_n}(\zeta))=(f_{\nu_0}(z_\nu+\rho_nu\zeta),\dots,f_{\nu_n}(z_\nu+\rho_nu\zeta))$$
on $N$. By using the argument principle, we get a sequence $\zeta_\nu\to \zeta_0$ such that for large enough $\nu$, we have 
$$\langle \widehat{g_\nu}(\zeta_\nu), H_{1, f_\nu}(z_\nu+\rho_\nu\zeta_\nu)\rangle=
\langle \widehat{f_\nu}(z_\nu+\rho_\nu\zeta_\nu), H_{1, f_\nu}(z_\nu+\rho_\nu\zeta_\nu)\rangle=0.$$
Let $(a^{(1)}_{\nu_0},\dots, a^{(1)}_{\nu_n})$ be a reduced representation of $\widetilde{H_{1_{f_\nu}}}$, then we obtain
 that 
\begin{equation}\label{eq: H_1f}
 \sum_{l=0}^n a^{(1)}_{\nu_l}(z_\nu+\rho_\nu\zeta_\nu)f_{\nu_l}=0.
\end{equation}
This implies that $f_\nu(z_\nu+\rho_\nu\zeta_\nu)\in H_{1, f_\nu}$, therefore, by using Condition (2), we have 
\begin{equation}\label{eq: cond 2}
|f_{\nu_0}(z_\nu+\rho_\nu\zeta_\nu)|\geq \varepsilon \|f_\nu(z_\nu+\rho_\nu\zeta_\nu)\|.
\end{equation}
Letting $\nu\to \infty$ in \eqref{eq: cond 2}, we get $|g_0(\zeta_0)|\geq \varepsilon\|g(\zeta_0)\|>0$. Which means $g_0(\zeta_0)\neq 0$. 
By the continuity of the function $g_0$ there exists  a neighbourhood $U$ such that $g_0(\zeta)\neq 0$ for all $\zeta \in U$. Thus, 
$g_{\nu_0}(\zeta)\neq 0$ on $U$ for large enough $\nu$. After renumbering again, we may assume that for each $\nu$,
\begin{equation*}
(g^2_{\nu_0}(\zeta), W(g_{\nu_0}, g_{\nu_1})(\zeta),\dots, W(g_{\nu_0}, g_{\nu_n})(\zeta))
\end{equation*} 
is a reduced representation of $\nabla g_\nu(\zeta)$ on $U$. Note that $f_{\nu_0}(z_\nu+\rho_\nu\zeta_\nu)\neq 0$, therefore, we have 
\begin{equation*}
\nabla f_\nu(z_\nu+\rho_\nu\zeta_\nu)=[f^2_{\nu_0}(z_\nu+\rho_\nu\zeta_\nu)\,:\,W(f_{\nu_0}, f_{\nu_1})(z_n+\rho_\nu\zeta_\nu)\,:\,\dots\,:\,
W(f_{\nu_0}, f_{\nu_n})(z_\nu+\rho_\nu\zeta_\nu)].
\end{equation*}
By Condition (1), $f_\nu(z_\nu+\rho_\nu\zeta_\nu)\in H_{1,f_\nu}$ if and only if 
$\nabla f_\nu(z_\nu+\rho_\nu\zeta_\nu)\in H_{1, f_\nu}$. Therefore, we obtain that
\begin{equation*}
W(f_{\nu_0},f_{\nu_l})(z_\nu+\rho_\nu\zeta_\nu)=f_{\nu_0}(z_\nu+\rho_\nu\zeta_\nu)\cdot f_{\nu_l}(z_\nu+\rho_\nu\zeta_\nu),
\end{equation*}
for $l=1, \dots, n$. Hence,
\begin{align}
|\langle\nabla f_\nu(z_\nu+\rho_\nu\zeta_\nu), H_{1, f_\nu}(z_\nu+\rho_\nu\zeta_\nu)\rangle|&=
|f_{\nu_0}(z_\nu+\rho_\nu\zeta_\nu)||\langle f_\nu(z_\nu+\rho_\nu\zeta_\nu), H_{1, f_\nu}(z_\nu+\rho_\nu\zeta_\nu) \rangle|\notag\\
&\leq |f_{\nu_0}(z_\nu+\rho_\nu\zeta_\nu)|\|H_{1,f_\nu}\|\|f_\nu(z_\nu+\rho_\nu\zeta_\nu)\|\notag\\
&\leq \frac{|f_{\nu_0}(z_\nu+\rho_\nu\zeta_\nu)|^2\|H_{1,f_\nu}\|}{\varepsilon}.\label{eq: cond 3}
\end{align}
Similarly, if $f_\nu(w_0)\in \displaystyle\cup_{j=1}^{2n+1}H_{j,f_\nu}$ then using
 Condition (2) of Theorem we have $f_{\nu_0}(w_0)\neq 0$. As above, we can prove that 
\begin{equation}\label{eq: cond 3-1}
|\langle\nabla f_\nu(w_0), H_{1, f_\nu}(w_0)\rangle|\leq  \frac{|f_{\nu_0}(w_0)|^2\|H_{1,f_\nu}\|}{\varepsilon}.
\end{equation}
Note that, by using \eqref{eq: H_1f}, we have $f_\nu(z_\nu+\rho_\nu\zeta_\nu)\in H_{1, f_\nu}$, therefore by \eqref{eq: cond 3-1}, we get
\begin{align*}
|\langle\nabla f_\nu(z_\nu+\rho_\nu\zeta_\nu), H_{j, f_\nu}(z_\nu+\rho_\nu\zeta_\nu)\rangle|&= 
|a^{(j)}_{\nu_0}(z_\nu+\rho_\nu\zeta_\nu)f_{\nu_0}^2(z_\nu+\rho_\nu\zeta_\nu)+ \\
&\qquad
a^{(j)}_{\nu_1}(z_\nu+\rho_\nu\zeta_\nu)W(f_{\nu_0}, f_{\nu_1})(z_\nu+\rho_\nu\zeta_\nu)+\\
& \qquad \dots+ a^{(j)}_{\nu_n}(z_\nu+\rho_\nu\zeta_\nu)W(f_{\nu_0}, f_{\nu_n})(z_\nu+\rho_\nu\zeta_\nu)|\\
&\leq \frac{\|f_{\nu_0}^2(z_\nu+\rho_\nu\zeta_\nu)\|\|H_{j,f_\nu}(z_\nu+\rho_\nu\zeta_\nu)\|}{\varepsilon}.
\end{align*}
This implies, for all $j=2, \dots, 2n+1$, that 
\begin{equation*}
\Big| a^{(j)}_{\nu_0}(z_\nu+\rho_\nu\zeta_\nu)+
\sum_{l=1}^n \frac{ a^{(j)}_{\nu_l}(z_\nu+\rho_\nu\zeta_\nu)W(g_{\nu_0}, g_{\nu_l})(\zeta_\nu)}{\rho_\nu g_{\nu_0}^2(\zeta_\nu)}\Big|\leq 
\frac{\|H_{1,f_\nu}(z_\nu+\rho_\nu\zeta_\nu)\|}{\varepsilon}.
\end{equation*}
Which gives , for all $j=2, \dots, 2n+1$, that 
\begin{equation*}
\Big|\rho_\nu a^{(j)}_{\nu_0}(z_\nu+\rho_\nu\zeta_\nu)+
\sum_{l=1}^n \frac{ a^{(j)}_{\nu_l}(z_\nu+\rho_\nu\zeta_\nu)W(g_{\nu_0}, g_{\nu_l})(\zeta_\nu)}{ g_{\nu_0}^2(\zeta_\nu)}\Big|\leq 
\frac{\rho_\nu\|H_{1,f_\nu}(z_\nu+\rho_\nu\zeta_\nu)\|}{\varepsilon}.
\end{equation*}
Letting $\nu\to \infty$, we obtain 
$\displaystyle\sum_{l=1}^n \frac{ a^{(j)}_{\nu_l}(z_0)W(g_{0}, g_{l})(\zeta_0)}{ g_{0}^2(\zeta_0)}=0$. Set
\begin{equation*}
\varphi_j(\zeta):= \sum_{l=1}^n \frac{ a^{(j)}_{\nu_l}(z_0)W(g_{0}, g_{l})(\zeta)}{ g_{0}^2(\zeta)},  \quad \zeta \in U.
\end{equation*}
\smallskip

We now proceed as in  \cite[ Proof of Theorem~2.1]{YPY 15} to observe that there are at most $n$ hyperplanes in 
$\{H_2, \dots, H_{2n+1}\}$ such that
$\varphi_j(\zeta)\equiv0$.  Suppose it is not true, and assume on the contrary that $\varphi_2(\zeta)\equiv\dots\equiv\varphi_{n+2}(\zeta)\equiv0$.
Since 
 \begin{equation*}
\varphi_j(\zeta)=\Big(\sum_{l=1}^{n} a^{(j)}_{\nu_l}(z_0)\frac{g_l}{g_0}\Big)'\equiv 0,
\end{equation*}
there exist $n+1$ complex numbers $c_2, \dots, c_{n+2}$ such that $\displaystyle\sum_{l=1}^n a^{(j)}_{\nu_l}(z_0)\frac{g_l}{g_0}\equiv c_j$ for 
all $j=2, \dots, n+2$. Since $H_1(z_0), \dots, H_{2n+1}(z_0)$ are in general position in $\mathbb{P}^n$, we infer that the system of equations
\begin{equation*}
\sum_{l=1}^n a^{(j)}_{\nu_l}(z_0) x_l \equiv c_j, \qquad \text{ for } j=2, \dots, n+2,
\end{equation*}
has no solutions or the solution is unique. This means that $g$ is  constant, which is a contradiction. Thus, we may assume that 
$\varphi_j(\zeta)\not\equiv 0$ for $j=2, \dots, n+1$. Then  for each $j\in \{2, \dots, n+1\}$, 
\begin{equation*}
\rho_\nu a_{\nu_0}^{(j)}(z_\nu+\rho_\nu\zeta)+\rho_\nu \sum_{l=1}^n a^{(j)}_{\nu_l}(z_\nu+\rho_\nu\zeta)\frac{W(f_{\nu_0}, f_{\nu_l})}{f_{\nu_0}^2}
(z_\nu+\rho_\nu\zeta)
\end{equation*}
converges uniformly to $\varphi_j(\zeta)$ on $U$. By the argument principle, again, we get a sequence $\zeta_\nu^*\to \zeta_0$ such that 
$$\langle \nabla f_\nu(z_\nu+\rho_\nu\zeta_\nu^*), H_{j,f_\nu}(z_\nu+\rho_\nu\zeta_\nu^*)\rangle =0.$$
Which further implies, by Condition (1) of Theorem, that 
$$\langle  f_\nu(z_\nu+\rho_\nu\zeta_\nu^*), H_{j,f_\nu}(z_\nu+\rho_\nu\zeta_\nu^*)\rangle =0.$$
Letting $\nu \to \infty$, we get $\langle  g(\zeta_0), H_{j}(z_0)\rangle =0$, for $j= 2,\dots, n+1$. Also from \eqref{eq: g zero}, we have 
$$\langle  g(\zeta_0), H_{j}(z_0)\rangle =0, \text{ for } j=1 ,2, \dots, n+1.$$
Which is a contradiction to the fact that $H_1(z_0), \dots, H_{2n+1}(z_0)$ are in general position in $\mathbb{P}^n$. Hence, $\mathcal{F}$ is normal. \end{proof}

\section*{Acknowledgements}\vspace{-1mm}
The second author is supported by the National Board for Higher Mathematics (NBHM) Post-doctoral fellowship (Ref. No. 0204/27/(20)/2023/R \& D-II/11900) and by a DST-FIST grant (grant no. DST FIST-2021 [TPN-700661]) at the Indian Institute of Science, Bangalore. The fourth author is supported by UGC Non-Net Fellowship.

\end{document}